\documentclass[reqno,11pt]{amsart}
\usepackage{tkz-euclide}
\usepackage{hyperref}
\usepackage{amsmath} 
\usepackage{amsfonts}
\usepackage{mathtools}
\usepackage{amssymb} 
\usepackage{mathrsfs} 
\usepackage{graphicx}  
\usepackage{bm}
\usepackage{tikz} 
\usepackage{centernot}
\usepackage{soul}
\usetikzlibrary{svg.path} 
\usepackage{longtable}
\usepackage{makecell}

\usepackage{geometry}
\usetikzlibrary{calc} 
\usepackage{xcolor}  
\usetikzlibrary{patterns}
\usetikzlibrary{arrows.meta} 
\usepackage{amsthm} 
\usepackage{float}
\usepackage{enumitem}
\setenumerate{ ref={\normalfont(\roman*)},label={\normalfont(\roman*)},itemsep=3pt}
\usepackage[english]{babel}
\usepackage[font=footnotesize, labelfont=bf]{ caption}
\usepackage{ esint }
\usepackage{stackengine,scalerel,graphicx}
\stackMath

\definecolor{trueblue}{rgb}{0.0, 0.45, 0.81}
\definecolor{truegreen}{rgb}{0.13, 0.55, 0.13}

\theoremstyle{plain}
\begingroup
\newtheorem{theorem}{Theorem}[section]
\newtheorem{lemma}[theorem]{Lemma}

\newtheorem{proposition}[theorem]{Proposition}
\newtheorem{corollary}[theorem]{Corollary}

\endgroup

\theoremstyle{definition}
\begingroup
\newtheorem{definition}[theorem]{Definition}

\endgroup

\renewcommand{\tilde}{\widetilde}

\DeclareMathOperator{\Res}{Res}

\renewcommand{\d}{ \mathrm{d}}

\numberwithin{equation}{section}
\newcommand{\N}{\mathbb{N}}

\newcommand{\E}{\mathcal{E}}
\newcommand{\R}{\mathbb{R}}
\newcommand{\C}{\mathbb{C}}

\renewcommand{\S}{{\mathbb{S}}}

\usepackage[normalem]{ulem}

\def \mb{\mathbb}

\def \n{\nabla}
\def \p{\partial}
\def \d{\,\mathrm{d}}
\def \mc{\mathcal}

\def \tp{\textup}

\begin{document}
	
\title
[Existence of degenerate bubbles]
{
On the existence of degenerate solutions\\ of the two-dimensional $H$-system
}
	
\author[Andr\'e Guerra]{Andr\'e Guerra} 
\address[Andr\'e Guerra]{Institute for Theoretical Studies, ETH Zürich, CLV, Clausiusstrasse 47, 8006 Zürich, Switzerland}
\email{andre.guerra@eth-its.ethz.ch}

\author[Xavier Lamy]{Xavier Lamy} 
\address[{Xavier Lamy}]{Institut de Mathématiques de Toulouse, Université Paul Sabatier, 118, route de Narbonne, F-31062 Toulouse Cedex 8, France}
\email{Xavier.Lamy@math.univ-toulouse.fr}	
	
\author{Konstantinos Zemas}
\address[Konstantinos Zemas]{Institute for Applied Mathematics, University of Bonn\\
Endenicher Allee 60, 53115 Bonn, Germany}
\email{zemas@iam.uni-bonn.de}
	


\begin{abstract}
{We consider entire solutions $\omega\in\dot H^1(\R^2;\R^3)$ of the $H$-system
\begin{align*}
\Delta\omega=2\omega_x\wedge\omega_y\,,
\end{align*}
which we refer to as \emph{bubbles}.
Surprisingly, and contrary to conjectures raised in the literature, we find that bubbles with degree at least three can be degenerate: the linearized $H$-system around a bubble can admit solutions that are not tangent to the smooth family of bubbles.  We then give a complete algebraic characterization of degenerate bubbles.}
\end{abstract} 
	
\maketitle
\thispagestyle{empty}


\section{Introduction}

In this short note we study critical points of the functional 
$\mc E\colon \dot H^1(\R^2
;\R^3
)\to \R$ defined by
\begin{equation}\label{eq:H_functional}
\mathcal E(u):=\frac{1}{2}\int_{\R^2}|\n u|^2 \d \mc L^2+\frac 2 3\int_{\R^2} \langle u, u_x\wedge u_y  \rangle \d \mc L^2  \,,
\end{equation} 
where $\langle\cdot,\cdot\rangle$ denotes the standard inner product in $\R^3$\,, $\mc L^2$ the 2-dimensional Lebesgue measure and $u_x,u_y$ the partial derivatives of $u$ with respect to $x,y$ respectively. Both terms in \eqref{eq:H_functional} are \textit{conformally invariant} and so, after identifying $\overline{\R^2}\cong \mb S^2$ through stereographic projection,  the term
\begin{equation}
\label{eq:vol}
 \mc V(u):=\frac 1 3 \int_{\R^2} \langle u,u_x \wedge u_ y\rangle \d \mc L^2
\end{equation}
corresponds to the signed algebraic volume of the region enclosed by $u(\mb S^2)$, whenever $u$ is regular. 
For maps which are just in the homogeneous Sobolev space $\dot H^1(\R^2;\R^3
)$, one can define $ \mc V(u)$ through its continuous extension \cite{Wente1969}. 

The functional \eqref{eq:H_functional} is very classical, as it appears naturally in the study of constant mean curvature surfaces, see \textit{e.g.}\ \cite[
Section~III.5\color{black}]{Struwe2008} and the references therein. To see this connection, we note that 
the first variation $\mathcal E'\colon \dot{H}^{1}(\R^2;\R^3) \to \dot H^{-1}(\R^2;\R^3)$ is given by
\begin{equation}\label{eq:first_variation}
\mathcal E'(u)=-\Delta u+2u_x\wedge u_y\,,
\end{equation} 
and thus we arrive at the following:

\begin{definition}\label{def:bubbles}
A \textit{bubble} is a map $\omega \in \dot H^1(\R^2;\R^3)$ such that $\mc E'(\omega)=0$, \textit{i.e.},
\begin{equation}\label{eq:equation_for_bubbles}	
\Delta \omega=2\,\omega_x\wedge \omega_y \ \text{ in } \mathcal{D}'(\R^2)\,.
\end{equation}
\end{definition}

In other words, bubbles are simply entire finite-energy solutions of \eqref{eq:equation_for_bubbles}, which is known as the \textit{$H$-system}.  Any such solution is necessarily \textit{weakly conformal}, \textit{i.e.}, it satisfies
\begin{equation}\label{eq:omega_conf}
|\omega_x|^2-|\omega_y|^2 =\langle \omega_x,\omega_y\rangle =0 \  \text{ in } \R^2\,.
\end{equation}
Combined, equations \eqref{eq:equation_for_bubbles} and \eqref{eq:omega_conf} assert that $\omega$ is a (possibly branched) weakly conformal parametrization of a closed surface with  mean curvature identically 1. In fact, a classical theorem due to Hopf asserts that any such surface is a unit sphere. This can also be seen  as a consequence of the following classification result, due to Brezis and Coron \cite[Lemma A.1]{Brezis1985},  which describes completely the collection of bubbles:

\begin{theorem}[Classification of bubbles \cite{Brezis1985}]\label{thm:BC}
Let $\omega \in \dot H^1(\R^2;\R^3)$ be a bubble. Then there exist complex polynomials $P,Q\in \C[z]$ and a vector $b\in \R^3$ such that
\begin{equation}\label{form_of_omega}
\omega(z)=\pi\left(\frac{P(z)}{Q(z)}\right)+b\,,
\end{equation}
where $\pi\colon \C \to \S^2$ is the inverse stereographic projection, i.e.,
\begin{equation}
\label{eq:stereo}
\pi(z):=\left(\frac{2z}{|z|^2+1},\frac{|z|^2-1}{|z|^2+1}\right)\,,
\end{equation}
and where we identify $ z=(x,y)=x+i y$. Moreover, if
\begin{equation}
\label{eq:condPQ}
P/Q \text{ is irreducible}, \quad k:=\max \{\deg P,\deg Q \},
\end{equation}
then we have
$$\frac{1}{8\pi}\int_{\R^2} |\n \omega|^2  \d \mc L^2  = k\,.$$
\end{theorem}

We refer to $k\in \N$ as in Theorem \ref{thm:BC} as the \textit{degree} of $\omega$, since it coincides with the topological degree of $\omega$,  viewed as a map between two-dimensional unit spheres.

Theorem \ref{thm:BC} shows that the collection of bubbles can be seen as a disjoint union of the smooth, finite-dimensional manifolds
\begin{equation}\label{eq:M_k}
\mc M_k:=\left\{(P,Q,b)\in \C[z]\times \C[z]\times \R^3: \text{at least one of } P,Q \text{ is monic and \eqref{eq:condPQ} holds}\right\}\,.
\end{equation}
An important problem is to understand the behavior of $\mc E$ near its critical points. To be more concrete, for any bubble $\omega\in \mc M_k$, we regard the second variation of $\mc E$ at $\omega$ as a linear operator $\mc E''(\omega)\colon \dot H^1(\R^2;\R^3)\to \dot H^{-1}(\R^2 ;\R^3)$. 
 In view of \eqref{eq:first_variation} it is explicitly given by
\begin{equation}\label{eq:linearized_operator}
\mathcal{E}''(\omega)[u]:=\frac{\d}{\d t}\Big|_{t=0}\mathcal{E}'(\omega+tu)=-\Delta u+2(\omega_x\wedge u_y+u_x\wedge \omega_y)\,.	
	\end{equation}
Variations tangent to $\mc M_k$ at $\omega$ generate elements in the kernel of $\mc E''(\omega)$, so 
 that
$$\dim \ker \mc E''(\omega)\geq \dim \mc M_k = 4k+5\,.$$
This leads us to the following standard concept:

\begin{definition}
A bubble $\omega$ with degree $k$ is said to be \textit{degenerate} if $\dim \ker \mc E''(\omega)>4k+5\,,$ and it is said to be \textit{non-degenerate} otherwise.
\end{definition}

As each $\mc M_k$ is smooth, non-degeneracy is equivalent to the usual notion of integrability, 
see \textit{e.g.}\ \cite[Section~3.13]{Simon1996}.
Due to the conformal invariance of $\mc E$, we can and will regard each bubble as a map $\omega\colon \S^2\to \R^3$.

In \cite[Lemma 5.5]{Isobe2001a} it was shown that  bubbles with degree one are non-degenerate, see further \linebreak \cite[Appendix]{Chanillo05} and \cite[Section 3]{Musina04}, while in \cite[Theorem 1.1]{SireWeiZhengZhou23} it was shown that the \textit{standard $k$-bubble}, corresponding to the choice $P(z)=z^k$ and $Q(z)=1$ in \eqref{form_of_omega}, is non-degenerate as well.  These works also raise the conjecture that all bubbles should be non-degenerate, cf.\ \cite[page 190]{Chanillo05} and \cite[page 4]{SireWeiZhengZhou23}. 
Surprisingly, in this note we observe that although this is generically the case, for $k\geq 3$ there is an \textit{exceptional set of degenerate bubbles} in $\mc M_k$:

\begin{theorem}[Characterization of degenerate bubbles]\label{thm:charact}
Let $\omega\colon \S^2\to \R^3$ be a bubble as in \eqref{form_of_omega} whose set of branch points is given by
$$\{|\n \omega| = 0 \} =:  \{p_1, \dots,p_n\}.$$
Let $z$ be a conformal coordinate on $\S^2$ with respect to which none of the branch points is 
 $\infty$. 
 Then $\omega$ is degenerate if and only if there exists a non-zero polynomial $R\in \C[z]$ with $\deg R \leq n-4$ such that the meromorphic function $h\colon \C\to \C$ defined by
\begin{equation*}
h(z) := \frac{R(z)}{(z-p_1)\dots(z-p_n)}
\end{equation*}
solves the algebraic system of equations
$$\Res_{p_j} \bigg(\frac{h}{(P/Q)'}\bigg) = 0 \quad \text{for } j\in \{1,\dots, n\}\,.$$
\end{theorem}

We note that Theorem \ref{thm:charact} implies that the set of degenerate bubbles of degree $k$ has (real) codimension at least two in $\mc M_k$, see the comment after \cite[Theorem 2]{EjiriKotani92}.
The above characterization follows from a correspondence between nontrivial elements in the kernel of $\mathcal E''[\omega]$ and solutions of a well-studied Schrödinger equation on $\S^2$, \textit{cf.}\ \cite{Ejiri94,Ejiri98,EjiriKotani92,EjiriKotani93,Kotani97,MontielRos91,Nayatani93}.  As immediate consequences of Theorem \ref{thm:charact} we obtain the following:

\begin{corollary}[Examples of non-degenerate bubbles]\label{cor:exnondeg}
Let $\omega$ be a degree $k$ bubble as in \eqref{form_of_omega}. If $\omega$ is degenerate then it has at least 4 branch points. In particular:
\begin{enumerate}
\item\label{it:deg2} if $k\leq 2$ then $\omega$ is non-degenerate\,;
\item\label{it:standard} the standard $k$-bubble corresponding to 
$$P(z)=z^k, \quad Q(z)=1\,,$$ is non-degenerate for all $k\in \N$.
\end{enumerate}
\end{corollary}

Other non-degenerate examples can be inferred from  \cite{MontielRos91}, see \textit{e.g.}\ \cite[Corollary 15]{MontielRos91}.

\begin{corollary}[Examples of degenerate bubbles]\label{cor:exdeg}
Let $\omega$ be a degree $k$ bubble as in \eqref{form_of_omega}. If $k=3$ then $\omega$ is degenerate if and only if
$$P(z) = z^3+2, \quad Q(z) = z\,,$$
up to a M\"obius transformation of the sphere.
\end{corollary}

An interesting question that can subsequently be pursued is to prove optimal \L ojasiewicz inequalities for degenerate bubbles.  There are only a few examples of variational problems where this is achieved  \cite{Frank2023,Frank2023a}. We also refer the reader to \cite{Malchiodi2020} for a related result, where \L ojasiewicz inequalities near the simplest possible bubble tree in a surface with positive genus were obtained.   Such inequalities determine in a precise way the leading order behavior of $\mc E$ near a bubble, and this information is useful in arguments involving blow-up or perturbative analysis, see \textit{e.g.}\ \cite{Brezis1985,Caldiroli04,Caldiroli04reduc,Caldiroli2006,Chanillo05,Isobe2001a, Isobe2001b,SireWeiZhengZhou23,Struwe1985}.

\subsection*{Acknowledgments}

AG was supported by Dr.\ Max R\"ossler, the Walter H\"afner Foundation and ETH Z\"urich Foundation. 
 XL was supported by 
the ANR project ANR-22-CE40-0006. 
KZ was supported by the Deutsche Forschungsgemeinschaft (DFG, German Research Foundation) through the Sonderforschungsbereich (SFB) 1060 and the Hausdorff Center for Mathematics (HCM) under Germany's Excellence Strategy -EXC-2047/1-390685813.

\section{Characterization of degenerate bubbles}

As explained in the introduction, the non-degeneracy of bubbles (cf.\ Definition \ref{def:bubbles})  is characterized through the study of the kernel of the linearized operator $\mathcal{E}''(\omega)\colon\dot H^{1}(\R^2;\R^3)\to \dot H^{-1}(\R^2;\R^3)$
defined in \eqref{eq:linearized_operator}.
In other words, we are interested in classifying the space of solutions $u\in \dot H^1(\R^2;\R^3)$ to
\begin{align}\label{eq:linHsyst}
\Delta u  =2(\omega_x\wedge u_y +u_x\wedge \omega_y)\  \text{ in } \mathcal{D}'(\R^2)\,
\end{align}
for an arbitrary bubble $\omega$. 
We note that, by elliptic regularity, any such solution is necessarily smooth.

Given a bubble $\omega$ as in \eqref{form_of_omega}, we can assume without loss of generality that $b=0$, since $\mc E$ is invariant under translations in the target. Then, due to the conformal invariance of $\mc E$, we can regard bubbles as maps $\omega \colon \mb S^2\to \mb S^2\subset \R^3$.  More explicitly, given a bubble $\omega$, we will write
\begin{equation}\label{eq:defomega}
\omega(z) = \pi(\varphi(z))\,, \qquad \varphi(z) := \frac{P(z)}{Q(z)}\,, \qquad \deg P \leq \deg Q\,,
\end{equation}
where $\pi$ is as in \eqref{eq:stereo} and $P,Q\in \C[z]$ for some conformal coordinate $z$ on $\mb S^2$; the assumption on the degrees of $P,Q$ can always be achieved by choosing a suitable coordinate $z$.  

We find it convenient to use complex notation, and henceforth we will write 
\[\p_z := (\p_x - i \p_y)/2\,, \quad \p_{\bar z} := (\p_x + i \p_y)/2\] for the usual Wirtinger derivatives.  Let the Euclidean inner product $\langle\cdot, \cdot\rangle$ in $\R^3$ be extended as a complex-bilinear form on $\C^3$.  Then the fact \eqref{eq:omega_conf} that each bubble is weakly conformal is expressed concisely as
\begin{equation}
\label{eq:holomw}
\langle \omega_z,\omega_z\rangle = 0 \ \text{in } \C\,,
\end{equation}
and the linearization of this equation leads us,
as in \cite[Section 2]{Gulliver89},
to the following:

\begin{definition}
Given a bubble $\omega\colon\S^2\to \S^2$ and a vector field $v\in \dot H^1(\S^2;\R^3)$, we say that $v$ is a \textit{conformal Jacobi field along}  $\omega$ if $\langle \omega_z,v_z\rangle = 0$\,.
\end{definition}

We can also express equations \eqref{eq:equation_for_bubbles} and \eqref{eq:linHsyst} respectively as
\begin{align}
\label{eq:complexHsys}
\omega_{z\bar z} & =  i\, \omega_{\bar z}\wedge \omega_z\,,\\[3pt]
u_{z \bar z} & = i (u_{\bar z}\wedge \omega_z + \omega_{\bar z} \wedge u_z)\,,
\label{eq:complexlinHsys}
\end{align}
in complex notation.
We then have the  following: 

\begin{lemma}\label{lem:confjacobi}
Let $u\colon \S^2\to \R^3$ solve \eqref{eq:complexlinHsys} for a bubble $\omega$. Then $u$ is a conformal Jacobi field.
\end{lemma}

\begin{proof}
We use \eqref{eq:complexHsys} and \eqref{eq:complexlinHsys} to compute
\begin{align*}
\p_{\bar z} \langle \omega_z,u_z\rangle
& = 
i \left(
\langle \omega_{\bar z} \wedge \omega_z, u_z\rangle  + \langle \omega_z, u_{\bar z} \wedge \omega_z+ \omega_{\bar z}\wedge u_z \rangle
\right) \\
&
= i \left(\langle \omega_{\bar z} \wedge \omega_z, u_z\rangle  + \langle \omega_z,  \omega_{\bar z}\wedge u_z \rangle \right)= 0\,,
\end{align*}
where the last equality is simply an algebraic identity resulting from $\langle a, b\wedge c\rangle = \det(a| b |c)$, where $(a| b |c)$ denotes the $3\times 3$ matrix with column vectors $a,b,c$ (in this order). Thus $\langle \omega_z, u_z\rangle$ is integrable and holomorphic, hence by Liouville's theorem it is zero.
\end{proof}

Note that each bubble $\omega$ as in \eqref{eq:defomega} is a \textit{harmonic map} into $\mb S^2$, \textit{i.e.},\
\begin{equation}\label{eq:harmonic_map}
\Delta \omega+|\nabla \omega|^2\omega=0 \quad \iff \quad \omega_{z \bar z} +|\omega_z|^2 \omega =0\,,
\end{equation}	
as can be checked directly by differentiating $|\omega|=1$, using the weak conformality  \eqref{eq:omega_conf} of $\omega$. Since the target is a 2-sphere, $\omega$ is an integrable harmonic map \cite{Gulliver89}; we note that this integrability does not hold for higher-dimensional spheres \cite{Lemaire2009}.  The following proposition is essentially a consequence of this integrability result.

\begin{proposition}[Decomposition]\label{prop:decomp}
Let $\omega\colon \S^2\to \S^2$ be a degree $k$ bubble
and
decompose maps 
 $u\in\ker\E''(\omega)$  according to
\begin{align*}
\ker \mc E''(\omega) 
&
= T(\omega)\oplus N(\omega)\,\omega\,,
\qquad
 u 
 =  \left[u -  \langle  u ,\omega \rangle \omega\right]+ \langle  u ,\omega \rangle \omega\,,
\end{align*}
where this decomposition is orthogonal pointwise
 with respect to the inner product of $\R^3$. The elements of $T(\omega)$ are generated by infinitesimal variations in the coefficients of $\varphi$ in \eqref{eq:defomega}, so that $\dim T(\omega) = 4k+2$, and 
\begin{equation}
\label{eq:defN}
N(\omega) := \left\{f\in C^\infty(\S^2;\R): \Delta f + |\nabla \omega|^2 f=0\right\}\,.
\end{equation}
\end{proposition}

\begin{proof}
Let us note that the equation defining $N(\omega)$ can be written in complex notation as
\begin{equation}
\label{eq:complexSchrod}
f_{z\bar z } + |\omega_z|^2 f=0\,.
\end{equation}
We now observe that $f\in N(\omega)$ if and only if the map $f\omega$ solves \eqref{eq:complexlinHsys}, since 
\begin{align*}
& (f \omega)_{z\bar z} -  i \big((f \omega)_{\bar z} \wedge \omega_z + \omega_{\bar z} \wedge (f \omega)_z\big) \\
& = f_{z \bar z}\, \omega - i f \omega_{\bar z} \wedge \omega_z + f_z \omega_{\bar z} + f_{\bar z} \omega_z - i \left(f_{\bar z} \omega \wedge \omega_z + f_z \omega_{\bar z} \wedge \omega\right)\\
&  = \left[ f_{z \bar z} +|\omega_z|^2 f \right]\omega\,,
\end{align*}
where  in the second line we used \eqref{eq:complexHsys} and in the last line we used the identities \[\omega_z = i \,\omega\wedge \omega_z\,,  \ \omega_{\bar z} = i\, \omega_{\bar z} \wedge \omega\,, \ \text{and } i \,\omega_z \wedge \omega_{\bar z} = |\omega_z|^2 \omega\,, \]
which follow from the fact that $\left(\frac{\sqrt 2\omega_x}{|\nabla \omega|},\frac{\sqrt 2\omega_y}{|\nabla \omega|}, -\omega\right)$ is a positively oriented orthonormal frame of $\R^3$.  
Hence, in view of \eqref{eq:complexlinHsys}, $N(\omega)\omega\subset \ker \mc E''(\omega)$ and the above decomposition follows. Note that, by elliptic regularity, any $f\in \dot H^1(\S^2; \R)$ solving \eqref{eq:complexSchrod} is actually smooth.

It remains to prove the characterization of $T(\omega)$. For this, we first note that since $|\omega|^2=1$ we have $\langle \omega_z,\omega\rangle = 0$. Thus, by \eqref{eq:holomw},  for any $f\in \dot H^1(\S^2;\R)$ the vector field $f\omega$ is a conformal Jacobi field along $\omega$.  By Lemma \ref{lem:confjacobi}, any $u\in \ker \mc E''(\omega)$ is also a conformal Jacobi field along $\omega$, and hence by linearity the same also holds for the map $\tilde u := u -  \langle u, \omega\rangle\omega$. Moreover, note that  $\tilde u(z) \in T_{\omega(z)}\S^2$. Let us assume that the polynomials $P,Q$ in \eqref{eq:defomega} satisfy $\det P \leq \deg Q$, the other case being entirely analogous. Applying \cite[Lemma 3]{Gulliver89},  we deduce that $\tilde u$ is integrable, \textit{i.e.},\ there are polynomials $A,B\in \C[z]$ such that
\begin{equation}
\label{eq:tildeu}
\tilde u = \frac{\d}{\d t}\bigg|_{t=0} \pi\left(\frac{P+t A}{Q+tB}\right),
\end{equation}
where $\deg (A)\leq \deg(P)$ and $\deg(B)<\det(Q)$.  Conversely, given polynomials $A,B$ satisfying these conditions, we can define a map $\tilde u\in T(\omega)$ through \eqref{eq:tildeu}. This completes the characterization and thus $\dim T(\omega)=4 k+2$.
\end{proof}

\begin{corollary}\label{cor:nullity}
We have $\dim N(\omega)\geq 3$ with equality if and only if $\omega$ is non-degenerate.
\end{corollary}

\begin{proof}
By \eqref{eq:harmonic_map},  the linear space
\begin{equation}
\label{eq:defL}
L(\omega) :=  \big\{ \langle \xi , \omega \rangle  :\xi \in \R^3\big\}
\end{equation}
is a subspace of $N(\omega)$, and thus $\dim N(\omega)\geq 3$.  By Proposition \ref{prop:decomp}, equality holds if and only if $\dim \ker \mc E''(\omega)=4k+5$, \textit{i.e.},\ if and only if $\omega$ is non-degenerate.
\end{proof}

We refer to $L(\omega)$ as the space of \textit{trivial} solutions to the Schr\"odinger equation \eqref{eq:complexSchrod}. Since $\omega$ is conformal,   \eqref{eq:complexSchrod} can be rewritten as
\begin{equation*}\label{eq:eigenvalue_eqn_sing_metric}
\Delta_{g_\omega}\tilde f+2 \tilde f=0\,, \ \text{where } g_\omega:=\omega^* g_0 \ \text{ and } \ \tilde f := f \circ \pi^{-1}\in \dot{H}^1(\S^2 ;\R)\,.
\end{equation*} 
Here $g_\omega$ is the pullback metric of the standard round metric $g_0$ on $\mathbb{S}^2$  induced by $\omega\colon \S^2\to \S^2$; the metric $g_\omega$ in general has \textit{conical singularities} at the branch points of $\omega$.  
Equation \eqref{eq:complexSchrod} has received a lot of attention in the last decades, and the space of solutions is completely understood, see \textit{e.g.}\ \cite{Ejiri94,Ejiri98,EjiriKotani92,EjiriKotani93,Kotani97,MontielRos91,Nayatani93}.  
Surprisingly,  for non-generic $\omega$  there are non-trivial solutions to \eqref{eq:complexSchrod}, \textit{i.e.}\ $\dim N(\omega)>3$; the space of such $\omega$ has codimension larger than one \cite[Proof of Theorem A]{EjiriKotani93}.  In general,  $\dim N(\omega)-3$ is \textit{even}, since $N(\omega)/L(\omega)$ is a complex vector space \cite[Proposition 18]{MontielRos91}, and so we can write $\dim N(\omega)=3+2d$ for $d\in \N_0$.  
The $2d$ extra directions can be interpreted as
arising from smooth one-parameter families of harmonic maps with values in higher dimensional spheres \cite[Theorem A]{Ejiri98}.

The following result provides a concrete algebraic characterization of those bubbles $\omega$ for which there are nontrivial elements in $N(\omega)$.

\begin{theorem}[\cite{EjiriKotani93,MontielRos91}]\label{thm:MR}
Let $\omega$ be a bubble as in \eqref{eq:defomega},  let  $\{p_1, \dots, p_n\}\subset \S^2$ be its set of branch points  and $z$ a conformal coordinate on $\mathbb{S}^2$  with respect to which none of the branch points is $\infty$. 
There is a linear bijection between the space $N(\omega)/ L(\omega)$ of non-trivial solutions to \eqref{eq:complexSchrod} and the space of non-zero polynomials $R\in \C[z]$ with $\deg R \leq n-4$ and such that 
\begin{equation}
\label{eq:condsR}
\Res_{p_j} \left(h/\varphi'\right)= 0 \text{ for all } j\in \{1, \dots, n\}\,, \ \text{where } h(z):= \frac{R(z)}{(z-p_1)\dots(z-p_n)}\,.
\end{equation}
\end{theorem}

The proof of Theorem \ref{thm:MR} in the above references is carried out in
 a more general context.   In order to keep the exposition mostly self-contained, we include here a direct proof of their result in our setting, following essentially the arguments in \cite{MontielRos91}.

The main idea behind Theorem \ref{thm:MR} is to use the \textit{Gauss parametrization} \cite{Dajczer1985}. A surface $S\subset \R^3$ with invertible Gauss map $\omega$ is locally parametrized by an immersion
\begin{equation}
\label{eq:defX}
X=f\omega + \nabla_{g_\omega} f = f\omega +\frac{1}{|\omega_z|^2}\left(f_z\omega_{\bar z}+f_{\bar z}\omega_z\right),
\end{equation}
where $f=\langle X, \omega\rangle$ is the \textit{support function} of $S$. Thus, in the setting of Theorem \ref{thm:MR},  we have a local correspondence between functions $f\colon \mb S^2\setminus \{p_1, \dots, p_n\}\to \R$ and punctured surfaces with generalized Gauss map $\omega$.  The next lemma gives a geometric interpretation of the condition $f\in N(\omega)$ (recall \eqref{eq:defN}) in terms of the corresponding immersion defined through \eqref{eq:defX}:

\begin{lemma}\label{lemma:gaussparam}
Let $\omega$ be a bubble with branch points $\{p_1, \dots, p_n\}\subset \S^2$,  let $f\in C^\infty(\S^2;\R)$ and let $X\colon \S^2\setminus \{p_1, \dots, p_n\}\to \R^3$ be defined as in \eqref{eq:defX}. Then 
$$f\in N(\omega) \quad \iff \quad 
X \text{ is minimal}\,, \text{  i.e., it is weakly conformal and harmonic}\,. 
$$
\end{lemma}

\begin{proof}
Using \eqref{eq:holomw} and $|\omega|=1$,  after differentiation we deduce that
\begin{equation}
\label{eq:iii}
\omega_{zz} =\big[\log(|\omega_z|^2)\big]_z \,\omega_z\,.
\end{equation}
An elementary but lengthy calculation, using \eqref{eq:iii}, then shows that
\begin{equation}
\label{eq:iv}
\big[\log(|\omega_z|^2)\big]_{z\bar z} + |\omega_z|^2=0\,.
\end{equation}
Note also that  $\p_z (\frac{1}{|\omega_z|^2}) = - \frac{\log(|\omega_z|^2)_z}{|\omega_z|^2}$; this, combined with \eqref{eq:harmonic_map} and \eqref{eq:iv}, yields
\begin{equation}
\label{eq:X_z}
X_z = \left(f_{z \bar z} + |\omega_z|^2 f\right)\frac{\omega_z}{|\omega_z|^2} + h_f  \frac{\omega_{\bar z}}{|\omega_z|^2}\,,
\end{equation}
where we set
\begin{align}\label{eq:defh}
h_f:=f_{zz}-\big[\log (|\omega_z|^2)\big]_z f_z\,.
\end{align}
We can then further compute
\begin{equation}
\label{eq:confX}
\langle X_z, X_z \rangle = 2\frac{h_f}{|\omega_z|^2} \left(f_{z \bar z} + |\omega_z|^2 f\right)\,.
\end{equation}

The  lemma follows from the above identities using elementary calculations. If $f\in N(\omega)$ then clearly $X$ is weakly conformal by \eqref{eq:confX} and one can also check that $X_{z \bar z}=0$, thus $X$ is minimal.  Conversely, if $X$ is minimal then $f\in N(\omega)$: this can be deduced from  \eqref{eq:X_z}  since 
$$\langle X_{z \bar z}, \omega\rangle = -\left(f_{z \bar z} + |\omega_z|^2 f\right)\,,$$
so if $X$ is harmonic then $f\in N(\omega)$.
\end{proof}

It will be useful to convert between $\omega$  and $\varphi$ through the stereographic projection, as in \eqref{eq:defomega}.  
Using the fact that $\varphi$ is meromorphic, one can verify that
\begin{equation}
\label{eq:derivatives}
\frac{\omega_{\bar z}}{|\omega_z|^2} = \frac{1}{\varphi'}\left(\frac{1-\varphi^2}{2},\frac{i(1+\varphi^2)}{2},\varphi\right)\,.
\end{equation}
The next lemma is a simple consequence of the above calculations.

\begin{lemma}\label{lemma:h}
For any conformal coordinate $z$ on $ \mathbb S^2$ and $f\in N(\omega)$, let $h_f\colon \mb S^2\setminus \{p_1,\dots,p_n\}\to \R$ be defined as in \eqref{eq:defh}.
Then:
\begin{enumerate}
\item \label{it:hismero}$h_f$ is meromorphic, i.e.,\ $\p_{\bar z} \,h_f=0$ on $\mathbb S^2\setminus \lbrace p_1,\ldots,p_n\rbrace$, and its poles are simple;
\item\label{it:zeroresidue} the meromorphic function $h_f/\varphi'$, where $\varphi=P/Q$, has zero residue at each pole $p_j$.
\end{enumerate}
\end{lemma}
\begin{proof}
For \ref{it:hismero} we note that, since $f\in N(\omega)$, from \eqref{eq:X_z} and \eqref{eq:derivatives} we obtain
\begin{equation}
\label{eq:X_z-varphi}
X_z =h_f  \frac{\omega_{\bar z}}{|\omega_z|^2}= \frac{h_f}{\varphi'} \left(\frac{1-\varphi^2}{2},\frac{i(1+\varphi^2)}{2},\varphi\right)\,.
\end{equation}
Thus $\p_{\bar z} h_f =0$, since  $\varphi$ is meromorphic and $X_{z \bar z}=0$  according to Lemma \ref{lemma:gaussparam}. Moreover, the poles of $h_f$ are simple, since  in \eqref{eq:defh}  $f$ is smooth and $|\omega_z|$ vanishes to finite order at each branch point.  
Claim  \ref{it:zeroresidue} follows again from \eqref{eq:X_z-varphi},  which implies that $h_f/\varphi'\d z= (X^1-i X^2)_z \d z$: the latter differential has zero integral along any loop, and hence zero residue at each pole.
\end{proof}

\begin{corollary}\label{cor:h}
For $f\in N(\omega)$
and any choice
of conformal coordinate $z$ 
 with respect to which none of the branch points is $\infty$,
 there exists $R\in \mathbb C[z]$ of degree at most $n-4$ such that
\begin{align}
\label{eq:Rhf}
h_f(z)=\frac{R(z)}{(z-p_1)\cdots (z-p_n)}\,.
\end{align}
\end{corollary}
\begin{proof}
Applying a translation in the complex plane if necessary, we may without restriction fix a conformal coordinate $w$ in which none of the poles of $\varphi$ is $0$ or $\infty$.
By Lemma \ref{lemma:h}\ref{it:hismero},  there exists an entire function $g$ and $a_1,\ldots,a_n\in\C$ such that
\begin{align*}
h_f(w)=g(w) +\sum_{j=1}^n \frac{a_j}{w-p_j}\,.
\end{align*}
If $w=w(z)$ is a conformal change of variables then the meromorphic quadratic differential  $h_f(w)(\tp d w)^2$ pulls back to $h_f(w(z)) (\tp d w/\tp d z)^2 (\tp d z)^2.$
Thus, when $w(z)=1/z$,   we obtain a new meromorphic function
\begin{align*}
h_f(w(z)) \bigg(\frac{\tp d w}{\tp d z}\bigg)^2 = \frac{1}{z^4}g(1/z) +\frac{1}{z^4}\sum_{j=1}^n \frac{a_j}{1/z -p_j}
=\frac{1}{z^4}g(1/z) +\frac{1}{z^3}\frac{\tilde R(z)}{\prod_{j=1}^{ n}(z-1/p_j)}\,,
\end{align*}
for some $\tilde R\in\mathbb C[z]$ with $\deg \tilde R\leq n-1$.
In these new coordinates,  Lemma \ref{lemma:h}\ref{it:hismero} implies that the entire function $g$ is zero and that $\tilde R(z)=z^3 R(z)$ for some $R\in\mathbb C[z]$ with $\deg R\leq n-4$.
\end{proof}

\begin{proof}[Proof of Theorem \ref{thm:MR}]

By Lemma \ref{lemma:h} and Corollary \ref{cor:h}, for each $f\in N(\omega)$ one can associate a function $h_f\colon \mb S^2\setminus \{p_1,\dots,p_n\}\to \R$ satisfying \eqref{eq:condsR}.  It is easy to see from \eqref{eq:iii}  and \eqref{eq:defh}  that if $f\in L(\omega)$ then $h_f=0$, \textit{i.e.},\ the corresponding polynomial $R$ in \eqref{eq:Rhf} is zero.

Conversely, given a function $h\colon \mb S^2\setminus \{p_1,\dots,p_n\}\to \R$ satisfying \eqref{eq:condsR}, one constructs an immersion $X_h\colon  \mb S^2\setminus \{p_1,\dots,p_n\}\to \R^3$ such that
$$\p_z X_h = h \frac{\omega_{\bar z}}{|\omega_z|^2}, $$
cf.\ \eqref{eq:X_z-varphi}. In fact, we construct $X_h$   by 
\begin{equation}
\label{eq:intX}
X_h = \textup{Re} \int \p_z X_h \d z\,,
\end{equation}
 where the path starts from a fixed point on $\mb S^2\setminus \{p_1,\dots,p_n\}$;  the no residue condition in \eqref{eq:condsR} guarantees that $X_h$ is independent of the choice of path. Moreover,  $X_h$ is unique up to addition of a constant vector in $\R^3$.  We claim that $\langle X_h, \omega\rangle \in N(\omega)$. Once this is shown the theorem follows, since the linear maps
$$N(\omega)\ni f\mapsto f\omega + \frac{1}{|\omega_z|^2} (f_z \omega_{\bar z} + f_{\bar z} \omega_z)\,, \qquad X\mapsto \langle X, \omega\rangle\in N(\omega)$$
are inverse to each other,  and hence 
$$N(\omega)/L(\omega)\ni f\mapsto h_f, \qquad h\mapsto \langle X_h,\omega\rangle\in N(\omega)/L(\omega)$$ are also inverses, since adding a constant to $X_h$ amounts to adding a trivial solution to $f$.

Thus, to complete the proof, it remains to show that $\langle X_h, \omega\rangle \in N(\omega)$.  To be precise, note that $\langle X_h, \omega\rangle$ is only defined in $\mb S^2\setminus \{p_1,\dots, p_n\}$ and that, in this punctured sphere, it is a solution of the Schr\"odinger equation \eqref{eq:complexSchrod}, since 
\begin{align*}
\p_{z\bar z} \langle X_h, \omega\rangle + |\omega_z|^2  \langle X_h, \omega\rangle 
& = \p_{\bar z} \left(\langle \p_z X_h, \omega\rangle + \langle X_h ,\omega_z\rangle \right) + |\omega_z|^2 \langle X_h, \omega\rangle\\
& =\langle X_h, \omega_{z\bar z}\rangle +  \langle X_h,  |\omega_z|^2 \omega\rangle=0\,.
\end{align*}
 The  last equality follows from \eqref{eq:harmonic_map} and to pass to the second line we used that 
$$\langle \partial_zX_h, \omega \rangle =\frac{h}{|\omega_z|^2}\langle \omega_{\bar z}, \omega\rangle=0\,,
$$
since $|\omega|=1$,
and that, since $X_h$ is a real vector field, 
$$\langle \p_{\bar z} X_h, \omega_z\rangle = \langle \p_{\bar z} \overline{X_h}, \omega_z\rangle = \langle \overline{\p_z X_h}, \omega_z\rangle=\frac{h}{|\omega_z|^2}\langle \omega_z, \omega_z\rangle=0\,, $$
cf.\  \eqref{eq:holomw}. Thus from elliptic regularity theory we deduce that $\langle X_h, \omega\rangle$ can be extended to a function in $N(\omega)$ once we show that $\langle X_h,\omega\rangle \in L^\infty(\mb S^2\setminus \{p_1, \dots ,p_n\}).$

To prove the boundedness of $\langle X_h,\omega\rangle$, fix $p_j$ for some $j=1,\dots,n,$ and choose a local conformal coordinate $z$ in a neighborhood of $p_j$ with $z(p_j)=0$.  Up to a rotation in $\R^3$, we can suppose that $\varphi(z) = z^{m_j}$ in this neighborhood, where $m_j>1$,  and so, from \eqref{eq:derivatives}, we see that
\begin{equation}\label{def:X_z}
\p_z X_h = \frac{h}{m_j} \left(\frac{z^{1-m_j} - z^{1+m_j}}{2}, \frac{i (z^{1-m_j} + z^{1+m_j})}{2} ,2z\right)\,.
\end{equation}
In these coordinates we may write $h$ as
$h(z)=\frac{c_0}{z}+\sum_{\ell=0}^\infty c_\ell z^\ell$, for some $(c_\ell)_{\ell\in \N}\in \C$, so that \eqref{eq:intX} and \eqref{def:X_z} imply that in a neighbourhood of 0,
$$X_h(z)=\frac{1}{2m_j(1-m_j)}\Big(c_0 \mathrm{Re}(z^{1-m_j})+o(z^{1-m_j}),c_0\mathrm{Im}(z^{1-m_j})+o(z^{1-m_j}),c_0(1-m_j)2\mathrm{Re}(z)+o(z)\Big)\,,$$
where $\underset{z\to 0}{\lim}\frac{o(z)}{z}=0$. By \eqref{eq:defomega}  and \eqref{eq:stereo}  the bubble $\omega$ can be expressed in these coordinates as
$$\omega(z)=\left(\frac{2z^{m_j}}{|z|^{2m_j}+1}, \frac{|z|^{2m_j}-1}{|z|^{2m_j}+1}\right)\,,$$ 
hence the fact that $\langle X_h, \omega\rangle$ is bounded near 0 follows 
from
the last two formulas.
\end{proof}

\begin{proof}[Proof of Theorem \ref{thm:charact}]
The result follows by combining Corollary \ref{cor:nullity} with Theorem \ref{thm:MR}.
\end{proof}

\begin{proof}[Proof of Corollary \ref{cor:exnondeg}]
The main claim follows from Theorem \ref{thm:charact}, since if $\omega$ is degenerate then there must exist a corresponding non-zero polynomial $R$ with $\deg R\leq n-4$, where $\omega$ has $n$ branch points, say $p_1, \dots, p_n$.  Hence if $\omega$ is degenerate we must have $n\geq 4$. This also immediately implies \ref{it:standard}. To prove \ref{it:deg2} note that if $m_j>1$ is the multiplicity of $\omega$ at $p_j$ (\textit{i.e.}\ if $m_j-1$ is the algebraic multiplicity of $p_j$ as a zero of $|\nabla \omega|$), then by the Riemann--Hurwitz formula  we have 
\begin{equation}
\label{eq:RH}
2(k-1) = \sum_{j=1}^n (m_j -1)\,,
\end{equation}
where $k$ is the degree of $\omega$. Thus $\omega$ can have at least 4 branch points only if $k\geq  3 $\,. 
\end{proof}

\begin{proof}[Proof of Corollary \ref{cor:exdeg}]
Note that, by \eqref{eq:RH}, a degree 3 bubble can only have at most 4 different branch points, so by Corollary \ref{cor:exnondeg} if $\omega$ is degenerate then indeed it must have exactly 4 different branch points $p_1,\dots, p_4$. By Theorem \ref{thm:charact}, we see that $\omega$ is degenerate if and only if
$$\Res_{p_j} \left(1/\varphi'\right) = 0\quad \text{for } j=1, 2,3,4\,,$$
where $\varphi = P/Q$ is as in \eqref{eq:defomega}. Elementary computations then yield the conclusion, see \cite[page 171]{MontielRos91} for further details.
\end{proof}

\bibliographystyle{acm}
\bibliography{ref_Hbubbles}

\end{document}